\documentclass[twoside,11pt]{article}
\usepackage{graphicx}
\usepackage{color}
\usepackage{amsmath,amssymb,amsthm,amsfonts}
\usepackage{graphicx}
\usepackage{color}
\usepackage{multicol}
\def\R{\mathbb{R}}
\makeatletter

\newcommand{\Rmnum}[1]{\expandafter\@slowromancap\romannumeral #1@}
\makeatother
\newcommand{\D}{\displaystyle}
\newtheorem{thm}{Theorem}[section]

\newtheorem{lemma}[thm]{Lemma}
\newtheorem{remark}{Remark}[section]
\newtheorem{theorem}[thm]{Theorem}

\usepackage[numbers,sort&compress]{natbib}
\def\O{\Omega}
\def\d{d}
\def\V{V}
\begin{document}

\title{{On the parabolic-elliptic Keller-Segel system with signal-dependent motilities: a paradigm for global boundedness and steady states}}
\author{\sc Zhi-An Wang
\thanks{%
Department of Applied Mathematics, The Hong Kong Polytechnic University, Hung Hom, Hong Kong; mawza@polyu.edu.hk. }}
\date{}
\maketitle

\begin{quote}
\textbf{Abstract}: This paper is concerned with a parabolic-elliptic Keller-Segel system where both diffusive and chemotactic coefficients (motility functions) depend on the chemical signal density. This system was originally proposed by Keller and Segel in \cite{KS-1971-JTB2} to describe the  aggregation phase of  {\it Dictyostelium discoideum}  cells in response to the secreted chemical signal cyclic adenosine monophosphate (cAMP), but the available analytical results are very limited by far.  Considering system in a bounded smooth domain with Neumann boundary conditions, we establish the global boundedness of solutions in any dimensions with suitable general conditions on the signal-dependent motility functions, which are applicable to a wide class of motility functions. The existence/nonexistence of non-constant steady states is studied and abundant  stationary profiles are found. Some open questions are outlined for further pursues. Our results demonstrate that the global boundedness and profile of stationary solutions to the Keller-Segel system with signal-dependent motilities depend on the decay rates of motility functions, space dimensions and the relation between the diffusive and chemotactic motilities, which makes the dynamics immensely wealthy.


\indent \textbf{AMS (2010) Subject Classification}: 35A01, 35B44, 35K57, 35Q92, 92C17
\end{quote}

\numberwithin{equation}{section}

\section{Introduction}
In this paper, we consider the following Keller-Segel (KS) system
\begin{eqnarray}\label{KS}
\begin{cases}
u_t=\nabla \cdot (\gamma(v)\nabla u-u \phi(v)\nabla v), &x\in\Omega, \ t>0, \\
\tau v_t=\d \Delta v+u-v, &x\in\Omega, \ t>0
\end{cases}
\end{eqnarray}
where $\Omega \subset \R^n (n\geq 1)$ is a bounded domain with smooth boundary, $u$ denotes the cell density and $v$ is the concentration of the chemical signal emitted by cells;  $\tau \in \{0,1\}$ and $d>0$ is the chemical diffusion rate; $\gamma(v)>0$ and $\phi(v)$ are diffusive and chemotactic coefficients (called motility functions in the sequel), respectively, both of which depend on the chemical signal concentration.  The system \eqref{KS} was derived by Keller and Segel in \cite{KS-1971-JTB2} to describe the  aggregation phase of  {\it Dictyostelium discoideum} (Dd)  cells in response to the chemical signal cyclic adenosine monophosphate (cAMP) secreted by Dd cells, where the motility functions $\gamma(v)$ and $\phi(v)$ are correlated  through the following relation
\begin{equation}\label{dchi}
\phi(v)=(\alpha-1)\gamma'(v),
\end{equation}
and $\alpha$ denotes the ratio of effective body length (i.e. distance between receptors) to step size. In a special case $\alpha=0$, namely the distance between receptors is zero and the chemotaxis occurs because of an undirected effect on activity due to the presence of a chemical sensed by a single receptor (cf. \cite[p.228]{KS-1971-JTB2}), the system \eqref{KS} is reduced to
\begin{eqnarray}\label{KS1}
\begin{cases}
u_t=\Delta(\gamma(v)u) &x\in\Omega, \ t>0,\\
\tau v_t=\d \Delta v+u-v &x\in\Omega,\ t>0.
\end{cases}
\end{eqnarray}
%

The Keller-Segel system \eqref{KS} with constant $\gamma(v)$ and $\phi(v)$ is called the minimal chemotaxis model (cf. \cite{Nanjundiah}), which has been extensively studied in the past few decades and a vast number of results have been obtained (cf. \cite{Blanchlet, BBTW-M3AS-2015,HP-JMB-2009, H-Review-1, Perthame, WWY} and references therein). In contrast,  the results of \eqref{KS} with non-constant $\gamma(v)$ and $\phi(v)$ are very few and to the best of our knowledge the existing results are available only for the special case $\phi(v)=-\gamma'(v)$, i.e. $\alpha=0$ in \eqref{dchi},  which simplifies the Keller-Segel system \eqref{KS} into \eqref{KS1}. Recently to describe the stripe pattern formation observed in the experiment of \cite{Liu-Science},  a so-called density-suppressed motility model was proposed in \cite{Fu-PRL-2011} as follows
\begin{equation}\label{KSN}
\begin{cases}
u_t=\Delta(\gamma(v)u)+\mu u(1-u),&x\in\Omega,\ t>0,\\
\tau v_t=\d\Delta v+u-v,&x\in\Omega,\ t>0.
\end{cases}
\end{equation}
with $\gamma'(v)<0$ and $\mu\geq 0$ denotes the intrinsic cell growth rate. Clearly the density-suppressed motility model \eqref{KSN} with $\mu=0$ coincides with the simplified KS model \eqref{KS1}. Indeed the density-suppressed motility has been previously used in the predator-prey system to describe the inhomogeneous co-existence distributions of ladybugs (predators) and
aphids (prey) populations in the field (see \cite{Kareiva} for modeling and \cite{JWEJAM} for mathematical analysis).

When the Neumann boundary conditions are imposed, namely $\partial_{\nu}u|_{\partial \Omega}=\partial_{\nu} v|_{\partial \Omega}=0$ where $\partial_\nu=\frac{\partial}{\partial \nu}$  with $\nu$ denoting the unit outward normal vector of $\partial \Omega$, there are some results available to \eqref{KS1} and \eqref{KSN}. For the system \eqref{KS1}, it was shown that globally bounded solutions exist in two dimensional spaces if the motility function $\gamma(v)\in C^3{([0,\infty)}\cap W^{1,\infty}(0,\infty))$ has both positive lower and upper bounds. It appears that this uniform boundedness assumption on $\gamma(v)$ is unnecessary for the global boundedness of solutions. For example, if $\gamma(v)=\frac{c_0}{v^k}$ (i.e. $\gamma(v)$ decays algebraically), it was proved in \cite{YK-AAM-2017} that global bounded solutions exist in all dimensions provided $c_0>0$ is small enough. Recently the global existence result was extended to the parabolic-elliptic case model (i.e. system \eqref{KS1} with $\tau=0$) in \cite{AY-Nonlinearity-2019} for any $0<k<\frac{2}{n-2}$ and $c_0>0$. When $\gamma(v)=\exp(-\chi v)$, a critical mass phenomenon has been shown to exist in \cite{Jin-Wang-PAMS} in two dimensions:  if $n=2$, there is a critical number $m=4\pi/\chi>0$ such that the solution of \eqref{KS1} with $\tau=d=1$ may blow up if the initial cell mass $\|u_0\|_{L^1(\O)}>m$ while global bounded solutions exist if $\|u_0\|_{L^1(\O)}<m$. This result was further refined in \cite{FujieJiang2020-2} showing that the blowup occurs at the infinity time.
For the system \eqref{KSN} with logistic growth (i.e. $\sigma>0$), the blowup in two dimensions was ruled out for a large class of motility function $\gamma(v)$.   Precisely, it is shown in \cite{JKW-SIAP-2018} that the system \eqref{KSN} has a unique global classical solution in two dimensional spaces if $\gamma(v)$ satisfies the following:
$\gamma(v)\in C^3([0,\infty)), \gamma(v)>0\ \  \mathrm{and}~\gamma'(v)<0 \ \ \mathrm{on} ~[0,\infty)$, $\lim\limits_{v \to \infty}\gamma(v)=0$ and $\lim\limits_{v \to \infty}\frac{\gamma'(v)}{\gamma(v)}$ exists. Moreover, the constant steady state $(1,1)$ of \eqref{KSN} is proved to be globally asymptotically stable if
$
\mu>\frac{K_0}{16}$ where $ K_0=\max\limits_{0\leq v \leq \infty}\frac{|\gamma'(v)|^2}{\gamma(v)}$. Recently, similar results have been extended to higher dimensions ($n\geq 3$) for large $\mu>0$ in \cite{WW} and to more relaxed conditions on $\gamma(v)$ in \cite{FujieJiang2020-1}. On the other hand, for small $\mu>0$, the existence/{\color{black}nonexistence} of nonconstant steady states  of \eqref{KSN} was rigorously established under some constraints on the parameters in \cite{MPW-PD-2019} and the periodic pulsating wave is analytically obtained by the multi-scale analysis.    When $\gamma(v)$ is a constant step-wise function, the dynamics of discontinuity interface was studied in \cite{SIK-EJAP-2019}.

By far, as recalled above, the study of the original Keller-Segel system \eqref{KS} was confined to the special case $\alpha=0$ (cf. \cite{YK-AAM-2017, AY-Nonlinearity-2019,Jin-Wang-PAMS}), namely the reduced system \eqref{KS1}, for some special form of $\gamma(v)$. The results for the case of $\alpha\ne 0$ remains entirely unknown. The objective of this paper is to establish the global boundedness of solutions to \eqref{KS} with suitable conditions on $\gamma(v)$ and $\phi(v)$ by keeping them as general as possible, and then apply the results a variety of $\gamma(v)$ and $\phi(v)$ including but beyond the relation \eqref{dchi}. As first step, we consider the parabolic-elliptic case of \eqref{KS}. That is, we consider the following problem
\begin{eqnarray}\label{KS-N}
\begin{cases}
u_t=\nabla \cdot (\gamma(v)\nabla u-u \phi(v)\nabla v), &x\in\Omega, \ t>0, \\
0=\d\Delta v+u-v, &x\in\Omega, \ t>0,\\
\partial_{\nu}u=\partial_{\nu} v=0 &x\in \partial \Omega,\\
u(x,0)=u_0 (x), &x\in \Omega.
\end{cases}
\end{eqnarray}
Except providing a general global boundedness result (see Theorem \ref{BS}), in this paper we develop a framework leading to the global boundedness of solutions by fully capturing the parabolic-elliptic structure to constructing a positive-definite quadratic form for gradients $\nabla u$ and $\nabla v$ to achieve necessary regularity/estimates (see Lemma \ref{priori}).  Although it is yet to be confirmed whether the results of \eqref{KS-N} can be wholly or partially carried over to the parabolic-parabolic case model of \eqref{KS} (i.e. $\tau=1$), they will be very instructive for the study of the parabolic-parabolic Keller-Segel model \eqref{KS} in the future.

The rest of this paper is organized as follows. In section 2, we state our main results and give some remarks on the implications/applications of our results. In section 3, we present the proof of our main results. The stationary solutions will be discussed in section 4. In final section 5, we shall summarize our results and outline a number of interesting questions open for further pursues.

\section{Statement of main results}
In this section, we shall state a general global existence result and present several specific applications. For the  motility functions $\gamma(v)$ and $\phi(v)$, we prescribe the following hypotheses
\begin{itemize}
\item[(H1)] $\gamma(v) \in C^2([0,\infty))$ and $\gamma(v)> 0$ for all $v \in [0,\infty)$.
\vspace{0.1cm}

\item[(H2)]
\begin{enumerate}
\item[(a)] $\phi(v) \in C^2([0,\infty))$, $\phi(v)\geq 0$ and $\phi'(v)<0$ for $v\in [0,\infty)$;
\item[(b)] $\lim\limits_{v \to \infty}v\phi(v)<\infty$ if $n>3$.
\end{enumerate}
\vspace{0.1cm}

\item[(H3)] $\inf\limits_{v\geq 0} \D\frac{\gamma(v) |\phi'(v)|}{|\phi(v)|^2}>\frac{n}{2}$.
\end{itemize}

The conditions (H1) and (H2) give the basic requirement on $\gamma(v)$ and $\phi(v)$, respectively, and (H3) imposes the constraint on the relation between $\gamma(v)$ and $\phi(v)$.  Note that the monotonicity of $\gamma(v)$ is not required, this is different from the existing results in \cite{YK-AAM-2017, AY-Nonlinearity-2019, Jin-Wang-PAMS}.

In the sequel, we say that $(u,v)$ is a classical solution to (\ref{KS-N}) in $\bar{\O}\times [0, T)$ for some $T\in(0,\infty]$ iff
$$u\in C(\bar{\Omega}\times[0,T))\cap C^{2,1}(\bar{\Omega}\times(0,T)), \ v \in C^{2,1}(\bar{\Omega}\times(0,T))$$
and $(u,v)$ satisfies equations \eqref{KS-N} pointwise.
Then our main results are stated in the following theorems.

\begin{theorem}\label{BS}
Let $\Omega\subset \R^n (n\geq 1)$ be a  bounded domain with smooth boundary. Assume $u_0  \gneqq 0$ and $u_0 \in W^{1,\infty}(\O)$. If one of the following holds

\begin{itemize}
\item[(i)] $n=1$, $\gamma(v)$ and $\phi(v)$ satisfy hypotheses (H1) and (H2)-(a);
\item[(ii)] $n\geq 2$, $\gamma(v)$ and $\phi(v)$ satisfy hypotheses (H1)-(H3) such that
\begin{equation}\label{condition}
\int_\O \phi(v)^{-p}dx<\infty \ \ \mathrm{for \ some} \ \ p>\frac{n}{2} \ \mathrm{and \ any} \ t>0,
\end{equation}
\end{itemize}
then the system (\ref{KS-N}) admits a unique classical solution
$(u,v) \in \bar{\O}\times [0,\infty)$ satisfying
\begin{equation}\label{bound}
\|u(\cdot,t)\|_{L^\infty}+\|v(\cdot,t)\|_{W^{1,\infty}}\leq C_0,
\end{equation}
where $C_0$ is a constant independent of $t$.
\end{theorem}

While assumptions (H1)-(H2) cover a wide range of motility functions $\gamma(v)$ and $\phi(v)$, we note that the global boundedness of solutions in one dimension ($n=1$) does not need the hypotheses (H2)-(b), (H3) and \eqref{condition} which comprise the main structural constraints on $\gamma(v)$ and $\phi(v)$ in multi-dimensions.
 If $\gamma(v)$ and $\phi(v)$ are explicitly given, the conditions (H3) and \eqref{condition} can be specified.
 Since the multi-dimensional problem is genuinely interesting in real world, below we assume $n\geq 2$ and explore the applications of Theorem \ref{BS} for motility functions with algebraic or exponential decay.

Before proceeding, we note by the integration of the first equation of \eqref{KS} that
\begin{equation}\label{L1}
\|u(t)\|_{L^1(\O)}=\|u_0\|_{L^1(\O)}:=m\ \mathrm{for\ all} \ t>0
\end{equation}
which indicates that the cell mass is conserved. Furthermore the local existence of classical solutions of \eqref{KS-N} can be obtained under hypotheses (H1) and (H2-(a)) only (see Lemma \ref{local}). Then from a known result of \cite[Corollary 2.3]{AY-Nonlinearity-2019} (see also \cite{F-JMAA-2015}), there is a positive constant $C(n, \Omega)>0$ such that
\begin{equation}\label{lower}
\inf _{x \in \Omega} v(x, t) \geqslant \eta, \ \mathrm{for \ all} \ 0<t<T_{\max}
\end{equation}
holds for a maximal existence time $T_{\max}\in (0, \infty]$, where $\eta=C(n, \Omega)\|u_0\|_{L^1(\O)}$. That is, the existence of priori positive number $\eta$ can be obtained under the hypotheses (H1)-(H2) without imposing  other conditions. Keeping this in mind, we consider the following two classes of motility functions $\gamma(v)$ and $\phi(v)$
\begin{equation}\tag{I}\label{rates1n}
\gamma(v)=\frac{\sigma_1}{v^{\lambda_1}},\ \ \phi(v)=\frac{\sigma_2}{v^{\lambda_2}}, \ \sigma_1, \sigma_2>0, \ \lambda_1>0, \lambda_2>1
\end{equation}
and
\begin{equation}\tag{II}\label{rates2n}
\gamma(v)=\exp(-\chi_1 v),\ \ \phi(v)=\delta \exp(-\chi_2 v), \ \ \ \chi_1>0, \chi_2>0, \delta>0.
\end{equation}

Then we have the following results.
\begin{theorem}\label{BSN}
Let $\Omega\subset \R^n (n\geq 2)$ be a  bounded domain with smooth boundary, and assume $u_0  \gneqq 0$ with $u_0 \in W^{1,\infty}(\O)$.  Then the system  \eqref{KS-N}  admits a unique classical solution $(u,v)$ in $\bar{\O}\times [0,\infty)$ satisfying \eqref{bound} if
\begin{itemize}
\item $\gamma(v)$ and $\phi(v)$ are given by \eqref{rates1n} with
\begin{equation}\label{con1}
\lambda_2\geq \lambda_1+1 \ \mathrm{and}\ \min\Big\{\frac{\lambda_2}{\lambda_2-1}, \frac{\sigma_1 \lambda_2}{\sigma_2} \eta^{\lambda_2-\lambda_1-1}\Big\}>\frac{n}{2}
\end{equation}
or
\item   $\gamma(v)$ and $\phi(v)$ given by \eqref{rates2n} with $n=2$ such that
\begin{equation}\label{con2}
\chi_2\geq \chi_1 \ \ \mathrm{and} \ \ \frac{n\delta}{2} \exp\{(\chi_1-\chi_2)\eta\}<\chi_2 <\frac{4\pi d}{m}.
\end{equation}
\end{itemize}
\end{theorem}
\begin{remark}\em{
We should remark that the results of Theorem \ref{BSN} are not simple applications of Theorem \ref{BS}. Indeed the validation of the key condition \eqref{condition} is quite technical and a lot additional efforts are needed depending on the specific form of $\phi(v)$ (see section 3.3).
}
\end{remark}
Note that the conditions $\lambda_2\geq \lambda_1+1$, $\frac{\sigma_1 \lambda_2}{\sigma_2} \eta^{\lambda_2-\lambda_1-1}>\frac{n}{2}$ in \eqref{con1} and conditions in \eqref{con2} stem from the hypothesis (H3), while the condition \eqref{condition} leads to $\frac{\lambda_2}{\lambda_2-1}>\frac{n}{2}$ in \eqref{con1} and requirement $n=2$ for \eqref{rates2n}.
Next we further explore the application of results in Theorem \ref{BSN} to the relation \eqref{dchi} originally derived by Keller and Segel in \cite{KS-1971-JTB2}.

\begin{theorem}\label{BSNN}
Let $\Omega\subset \R^n (n\geq 2)$ be a  bounded domain with smooth boundary and $u_0 \gneqq 0$ with  $u_0 \in W^{1,\infty}(\O)$. If $\gamma(v)$ and $\phi(v)$ satisfy the relation \eqref{dchi} with $\alpha<1$ and one of the following assumptions holds
\begin{itemize}
\item[(i)] $\gamma(v)=\frac{\sigma}{v^\lambda}$ ($\sigma>0$) such that
\begin{equation}\label{gcon}
0<\lambda< \begin{cases} \frac{2}{n-2} &\ \mathrm{if} \ 0\leq \alpha<1 \\
\frac{2}{n(1-\alpha)-2} &\ \mathrm{if} \ \alpha<0;  \end{cases}
\end{equation}
\item[(ii)] $\gamma(v)=\exp(-\chi v)$ with $n=2, \ \chi<\frac{4\pi d}{m}$ and  $0<\alpha<1$;
\end{itemize}
then the system \eqref{KS-N}  has a unique classical solution $(u,v)$ in $\bar{\O}\times [0,\infty)$ satisfying \eqref{bound}.
\end{theorem}


\begin{remark} \em{We have several remarks regarding the results of Theorem \ref{BSNN}.
\begin{enumerate}
\item With relation \eqref{dchi} and function $\gamma(v)$ with algebraic or exponential decay as in Theorem \ref{BSNN}, the lower bound value $\eta$ for $v$ does not play a role since $\lambda_2=\lambda_1+1$ or $\chi_2=\chi_1$.
\item If $\alpha=0$, the result of  Theorem \ref{BSNN} (i) recovers the global existence result of \cite{AY-Nonlinearity-2019}. When $\gamma(v)=e^{-\chi v}$ and $n=2$, it was shown recently in \cite{FujieJiang2020-1, Jin-Wang-PAMS} that the system \eqref{KS1} with $\tau=d=1$ possesses a critical mass $m_c=4\pi/\chi>0$ such that the solution may blow up if $\|u_0\|_{L^1(\O)}>m$ while globally exist if $\|u_0\|_{L^1(\O)}<m_c$. Our results in Theorem \ref{BSNN} (ii) extend the same global boundedness results to the case $0<\alpha<1$.
\end{enumerate}
}
\end{remark}

\begin{remark}\em{
It is worthwhile to note that the monotonicity of $\gamma(v)$ and relation \eqref{dchi} are not required in Theorem \ref{BS}, and hence the applications of our results are far more than those motility functions $\gamma(v)$ and $\phi(v)$ discussed above. For example, one can consider the following motility functions
$$\gamma(v)=\ln(v+1), \ \ \phi(v)=\frac{\sigma}{(v+1)^\lambda}\ (\mathrm{or} = e^{-\chi v})$$
and follow the results of Theorem \ref{BS} to find the appropriate conditions for the global boundedness of solutions.
}
\end{remark}

\section{Proof of Main Results}
In this section, we first give the local existence of solutions and recall some well-known results for later use. Then we derive a global boundedness criterion for the system \eqref{KS-N} and show a sufficient condition ensuring such criterion. Finally we proceed to prove our main results stated in Section 1. In the sequel, when appropriate, we use $c_i$ or $C_i$ ($i=1,2,\cdots$) to denote a generic positive constant varying in the context.

\subsection{Preliminaries}
 The local existence of solutions of \eqref{KS1} and \eqref{KSN} was proved in \cite{AY-Nonlinearity-2019} and \cite{JKW-SIAP-2018}, respectively, by the Schauder fixed point theorem, and the uniqueness was proved by a direct argument. We can employ the exact procedures as in \cite{AY-Nonlinearity-2019, JKW-SIAP-2018} with slight modifications to get the local existence and uniqueness of solutions to \eqref{KS-N}. The local existence can also be obtained by Amann's theorem on the triangular system (cf. \cite{A-Book-1993} or \cite{JWEJAM}).  Below we shall state the results only and omit the proof for brevity.

\begin{lemma}[Local existence]\label{local}
Let $\Omega$ be a bounded domain in $\R^n(n\geq 1)$ with smooth boundary and assume $\gamma(v)$ and $\phi(v)$ satisfy the hypotheses (H1) and (H2-(a)). If $u_0  \gneqq 0$ and  $u_0 \in W^{1,\infty}(\O)$, then there exist $T_{\max} \in (0,\infty]$ such that the problem \eqref{KS} has a unique classical solution $(u,v)\in [C^0(\bar{\Omega}\times[0,T_{\max}))\cap C^{2,1}(\bar{\Omega}\times(0,T_{\max}))] \times C^{2,1}(\bar{\Omega}\times(0,T_{\max}))$ satisfying $u,v>0$ in $\Omega \times (0,T_{\max})$. Moreover if $T_{\max}<\infty$, then
$$\lim\limits_{t \nearrow T_{\max}} \|u(\cdot, t)\|_{L^\infty(\Omega)} =\infty.$$
\end{lemma}

For convenience, we recall a well-known result below (cf. \cite{Brezis-Strauss}).
\begin{lemma}\label{regularity}
Let $\Omega$ be a bounded domain in $\R^n(n\geq 1)$ with smooth boundary and $u\in L^1(\O)$ be a non-negative function. If $v \geq 0$ satisfies
\begin{eqnarray*}
\begin{cases}
-\d \Delta v+v=u, &x\in\Omega,\\
\partial_{\nu} v=0, &x\in \partial \Omega,
\end{cases}
\end{eqnarray*}
then
\begin{eqnarray*}
v\in
\begin{cases}
L^\infty, \ & \ \text{if} \ n=1,\\
L^q(1\leq q<\infty), \ & \ \text{if} \ n=2,\\
L^r (1\leq r< {\frac{n}{n-2}}), \ & \ \text{if} \ n>2.
\end{cases}
\end{eqnarray*}
\end{lemma}


\begin{lemma}\label{Merle}
Let $\Omega$ be a bounded domain in $\R^2$ with smooth boundary. Consider the following problem
\begin{eqnarray*}
\begin{cases}
-\d \Delta v+v=u, &x\in\Omega,\\
\partial_{\nu} v=0, &x\in \partial \Omega
\end{cases}
\end{eqnarray*}
where $u\in L^1(\O)$ with $\|u\|_{L^1(\O)}=m$. If $0< \Lambda <4\pi d/m$, then there is a constant $C>0$ such that the solution of the above problem satisfies $$\int_\O e^{\Lambda v}dx\leq C.$$
\end{lemma}
\begin{proof}
The proof is inspired by \cite[Theorem 1]{Brezis-Merle-CPDE} (see also \cite[Theorem A.3]{Tao-Winkler-JDE}). For preciseness of our results, we present a proof similar to the one of \cite[Theorem A.3]{Tao-Winkler-JDE}.  Let $G(x,y)$ denote the Green's function of $d \Delta+1$ in $\Omega$ subject to the homogeneous Neumann boundary condition. Then it follows that (cf. \cite{Senba-Suzuki-ADM, Senba-Suzuki-2000})
\begin{eqnarray}\label{green}
|G(x,y)| \leq \frac{1}{2 \pi d}\ln \frac{1}{|x-y|}+K \ \ \mathrm{for \ all} \ x,y \in \Omega \ \mathrm{with} \ x\ne y
\end{eqnarray}
where $K$ is positive constant. Then $v$ can be represented as
$$v(x)=\int_\O G(x,y) u(y)dy$$
which yields from \eqref{green} that
$$v(x)\leq \int_\O \Big(\frac{1}{2\pi d}\ln\frac{1}{|x-y|}+K\Big)\cdot |u(y)|dy\leq \frac{1}{2\pi d} \int_\O \ln\frac{1}{|x-y|}\cdot |u(y)|dy+Km.$$
The by Jensen's inequality and Fubini's theorem, we have
\begin{eqnarray*}
\begin{aligned}
\int_{\Omega} e^{\Lambda v(x)} dx & \leq e^{\Lambda Km} \int_{\Omega} e^{\frac{\Lambda m}{2\pi d} \cdot \int_{\Omega} \ln \frac{1}{|x-y|} \cdot \frac{|u(y)|}{m} d y} d x \\
&\leq  A\int_{\Omega}\left(\int_{\Omega} e^{\frac{\Lambda m}{2\pi d} \cdot \ln \frac{1}{|x-y|}} \cdot \frac{|u(y)|}{m} d y\right) d x \\
&= A\int_{\Omega} \int_{\Omega}|x-y|^{-\frac{\Lambda m}{2\pi d}} \cdot \frac{|u(y)|}{m} d y d x \\
&\leq  A\int_{\Omega}\int_{\Omega}|x-y|^{-\frac{\Lambda m}{2\pi d}} \cdot \frac{|u(y)|}{m} d y d x \\
&=A \int_{\Omega}\left(\int_{\Omega}|x-y|^{-\frac{\Lambda m}{2\pi d}} d x\right) \cdot \frac{|u(y)|}{m} d y
\end{aligned}
\end{eqnarray*}
where $A=e^{\Lambda Km}$. Since $\O$ is bounded, if $\frac{\Lambda m}{2\pi d}<2$ (i.e. $\Lambda <4\pi d/m$), then there is a constant $c_0>0$  such that $\int_{\Omega}|x-y|^{-\frac{\Lambda m}{2\pi d}}dx<c_0$ and hence
\begin{eqnarray*}
\int_{\Omega} e^{\Lambda v(x)} dx \leq c_0 A \int_\O \frac{|u(y)|}{m} d y=c_0 A.
\end{eqnarray*}
This completes the proof.
\end{proof}

\begin{lemma}[Trudinger-Moser inequality \cite{Nagai-Funk}]\label{moser}
Let $\Omega$ be a bounded domain in $\R^n (n\geq 2)$ with smooth boundary. Then  for any $u\in W^{1,n}(\O)$ and any $\varepsilon>0$,   there exists a positive constant $C_\varepsilon$  depending on $\varepsilon$ and $\Omega$ such that
\begin{equation*}\label{T-M}
\int_\Omega \exp{|u|}dx\leq C_\varepsilon \exp\left\{\left(\frac{1}{\beta_n}+\varepsilon\right)\|\nabla u\|_{L^n(\O)}^n+\frac{1}{|\Omega|}\|u\|_{L^1(\O)}^n\right\}
\end{equation*}
where $\beta_n=n\big(\frac{n \alpha_n}{n-1}\big)^{n-1}$ and $\alpha_n=n\omega_{n-1}^{1/(n-1)}$ with $\omega_{n-1}$ denoting the $(n-1)$-dimensional surface area of the unit sphere in $\R^n$.
\end{lemma}

%

\subsection{A boundedness criterion}

\begin{lemma}\label{criterion}
Let the assumptions in Lemma \ref{local} hold. If there exists constant $C_0>0$ independent of $t$ such that the following inequality holds for any $0<t<T_{\max}$
\begin{equation}\label{Lp}
\|u(\cdot,t)\|_{L^p(\Omega)} \leq C_0\ \ \text{for some} \ \ p>\frac{n}{2},
\end{equation}
then the system (\ref{KS-N}) has a unique classical solution $(u,v)$ satisfying
\begin{equation*}
\|u(\cdot,t)\|_{L^\infty(\O)}+\|v(\cdot,t)\|_{W^{1,\infty}(\O)}\leq C,
\end{equation*}
where $C$ is a constant independent of $t$.
\end{lemma}

\begin{proof}
We first claim under \eqref{Lp} the following inequality holds
\begin{equation}\label{claim}
\|u\|_{L^p(\O)} \leq c_0 \ \ \mathrm{for \ all} \ \ t\in (0, T_{\max}), \ p>n
\end{equation}
for some constant $c_0>0$ independent of $t$. To this end, we multiply the first equation of \eqref{KS-N} by $ u^{p-1}$ $(p>1)$ and integrate the resulting equation by parts to get
\begin{eqnarray}\label{li1}
\begin{split}
\frac{1}{p} \frac{d}{dt} \int_{\Omega} u^p dx =& -(p-1)\int_{\Omega} u^{p-2} \nabla u(\gamma(v)\nabla u-u\phi(v)\nabla v)dx\\
=& -(p-1) \int_{\Omega} \gamma(v)u^{p-2} |\nabla u|^2 dx+ (p-1) \int_{\Omega} \phi(v)u^{p-1} \nabla u\nabla v dx.
\end{split}
\end{eqnarray}
Thanks to the elliptic regularity theorem applied to the second equation of \eqref{KS-N}, we have $v\in W^{2, p}(\O)$ given $u \in L^p(\O)$. Then by the Sobolev embedding and \eqref{Lp}, we find a constant $c_1>0$ such that
\begin{equation}\label{li1n}
\|v\|_{L^\infty(\O)} \leq c_1 \ \ \mathrm{for \ all} \ \ t\in (0, T_{\max}).
\end{equation}
Then we can find two constants $c_2, c_3>0$, thanks to \eqref{lower} and hypotheses (H1)-(H2), such that
$$\gamma(v)\geq c_2, \ \phi(v)\leq c_3.$$
Then it follows from \eqref{li1} that
\begin{eqnarray}\label{li2}
\frac{1}{p} \frac{d}{dt} \int_{\Omega} u^p dx+c_2(p-1) \int_{\Omega}u^{p-2} |\nabla u|^2 dx\leq c_3(p-1) \int_{\Omega} u^{p-1} \nabla u\nabla v dx.
\end{eqnarray}
By Young's inequality, we have
\begin{eqnarray*}\label{li3}
c_3(p-1) \int_{\Omega} u^{p-1} \nabla u\nabla v dx \leq \frac{c_2}{2}(p-1) \int_{\Omega}u^{p-2} |\nabla u|^2 dx+c_4 \int_\O u^p|\nabla v|^2dx
\end{eqnarray*}
which, substituted into \eqref{li2} along with the fact $u^{p-2}|\nabla u|^2=\frac{4}{p^2}|\nabla u^{\frac{p}{2}}|^2$, gives
\begin{eqnarray}\label{li4}
\frac{1}{p} \frac{d}{dt} \int_{\Omega} u^p dx+\frac{2c_2(p-1)}{p^2} \int_{\Omega} |\nabla u^{\frac{p}{2}}|^2 dx\leq c_4 \int_\O u^p|\nabla v|^2dx.
\end{eqnarray}
Next we estimate the term on the right hand side of \eqref{li4}. First the Young's inequality gives
\begin{eqnarray}\label{li5}
\int_\O u^p|\nabla v|^2dx \leq \int_\O u^{p+1}dx+\int_\O |\nabla v|^{2(p+1)}dx
\end{eqnarray}
where the last inequality follows from the elliptic regularity applied to the second equation of \eqref{KS-N}.
Moreover the Gagliardo-Nirenberg inequality with \eqref{li1n} leads to
\begin{eqnarray*}\label{li6}
\begin{aligned}
\|\nabla v\|_{L^{2(p+1)}(\O)}^{2(p+1)}\leq c_5 \|v\|_{W^{2,p+1}(\O)}^{p+1}\|v\|_{L^\infty(\O)}^{p+1}
\leq c_6\|v\|_{W^{2,p+1}(\O)}^{p+1}\leq c_7\|u\|_{L^{p+1}(\O)}^{p+1}.
\end{aligned}
\end{eqnarray*}
This along with \eqref{li5} updates \eqref{li4} as
\begin{eqnarray}\label{li7}
\frac{1}{p} \frac{d}{dt} \int_{\Omega} u^p dx+\frac{2c_2(p-1)}{p^2} \int_{\Omega} |\nabla u^{\frac{p}{2}}|^2 dx\leq c_8 \int_\O u^{p+1}dx+c_9.
\end{eqnarray}
Now adding $\frac{1}{p}\int_{\Omega} u^p dx$ to both sides of \eqref{li7} and using the fact
$$\frac{1}{p}\int_{\Omega} u^p dx \leq \int_\O u^{p+1}dx+c_p$$
for some constant $c_p>0$ by the Young's inequality, we have from \eqref{li7} that
\begin{eqnarray}\label{li8}
\frac{1}{p} \frac{d}{dt} \int_{\Omega} u^p dx+\frac{1}{p}\int_{\Omega} u^p dx+\frac{2c_2(p-1)}{p^2} \int_{\Omega} |\nabla u^{\frac{p}{2}}|^2 dx\leq c_{10} \int_\O u^{p+1}dx+c_{11}.
\end{eqnarray}
Next we employ the Gagliardo-Nirenberg inequality again to have
\begin{equation}\label{li9}
\int_{\Omega} u^{p+1}dx=\Big\|u^{\frac{p}{2}}\Big\|_{L^{\frac{2(p+1)}{p}}(\Omega)}^{\frac{2(p+1)}{p}} \leqslant C_{GN}\bigg(\left\|u^{\frac{p}{2}}\right\|_{L^{1}(\Omega)}^{\frac{2(p+1)}{p}(1-\theta)}\left\|\nabla u^{\frac{p}{2}}\right\|_{L^{2}(\Omega)}^{\frac{2(p+1)}{p} \theta}+\left\|u^{\frac{p}{2}}\right\|_{L^{1}(\Omega)}^{\frac{2(p+1)}{p}}\bigg)
\end{equation}
with $\theta=\frac{n}{n+2} \frac{p+2}{p+1} \in(0,1)$ due to $p>n$. By \eqref{Lp}, we know for $p>n$ it holds that
$$\|u^{\frac{p}{2}}\|_{L^1(\O)}=\|u(\cdot, t)\|_{L^{\frac{p}{2}}(\Omega)}^{\frac{p}{2}} \leqslant c_{12} \quad \mathrm{for\ all}\ \ t \in\left(0, T_{\max }\right)$$ which updates \eqref{li9} as
\begin{equation}\label{li10}
c_{10}\int_{\Omega} u^{p+1}dx\leq  c_{13}\bigg(\left\|\nabla u^{\frac{p}{2}}\right\|_{L^{2}(\Omega)}^{\frac{2(p+1)}{p} \theta}+1\bigg) \leq
\frac{c_2(p-1)}{p^2} \left\|\nabla u^{\frac{p}{2}}\right\|_{L^{2}(\Omega)}^2+c_{14}
\end{equation}
where we have used Young's inequality based on the fact $\frac{2(p+1)}{p} \theta=\frac{2(np+2n)}{pn+2p}<2$ due to $p>n$. Then substituting \eqref{li10} into \eqref{li8} gives
\begin{eqnarray*}\label{li11}
\frac{1}{p} \frac{d}{dt} \int_{\Omega} u^p dx+\frac{1}{p}\int_{\Omega} u^p dx\leq c_{15}
\end{eqnarray*}
which by the Gronwall's inequality proves the claim \eqref{claim}. Now with \eqref{claim} and the elliptic regularity theorem, we get from the second equation of \eqref{KS-N} that $v\in W^{2,p}(\O)\hookrightarrow C^{1,1-\frac{n}{p}}(\O)$ by the Sobolev embedding theorem. Hence there exists a constant $c_{16}>0$ independent of $t$ such that
\begin{equation}\label{W1V}
\|v\|_{W^{1,\infty}(\O)} \leq c_{16} \ \ \mathrm{for \ all} \ \ t\in (0, T_{\max}).
\end{equation}
Then we apply \eqref{W1V} into \eqref{li4} and obtain
\begin{eqnarray}\label{li14}
\frac{1}{p} \frac{d}{dt} \int_{\Omega} u^p dx+\frac{2c_2(p-1)}{p^2} \int_{\Omega} |\nabla u^{\frac{p}{2}}|^2 dx\leq c_{17} \int_\O u^pdx.
\end{eqnarray}
Starting from \eqref{li14}, we can utilize the standard Moser iteration (cf. \cite{Alikakos-1979}) to prove that $\|u\|_{L^{\infty}(\O)} \leq c_{18} \  \mathrm{for \ all} \ t\in (0, T_{\max})$ (e.g. see the proof of Theorem 2.1 in \cite{Tao-Wang}). We omit the details here for brevity. Then the standard elliptic regularity applied to the second equation of \eqref{KS-N} with $u\in L^\infty(\O)$ yields that $\|v\|_{W^{1,\infty}(\O)}\leq c_{19}$ for some constant $c_{19}>0$. This completes the proof of Lemma \ref{criterion}.
\end{proof}

By the result of Lemma \ref{criterion}, to prove our results, it is the key to drive the {\it priori} inequality \eqref{Lp}. When $n<2$, \eqref{Lp} directly holds true by taking $p=1$ due to \eqref{L1}.   In the following we hence assume $n\geq 2$ and proves a useful inequality to show \eqref{Lp}.

\begin{lemma}\label{priori}
Let $\Omega \subset \R^n(n\geq 2$) be a bounded domain with smooth boundary. Let $\gamma(v)$ and $\phi(v)$ satisfy hypotheses (H1)-(H3) and  $(u,v)$ be a classical solution obtained in Lemma \ref{local} with the maximal existence time $T_{\max}\in (0,\infty]$. Then there exists some $p>\frac{n}{2}$ such that
\begin{eqnarray}\label{pei}
\frac{d}{dt} \int_{\Omega} u^p dx+\int_\O u^pdx \leq c_0+c_1 \int_\O \phi(v)^{-p}dx \ \ \ \mathrm{for\ all} \ t\in (0, T_{\max})
\end{eqnarray}
where $c_0$ and $c_1$ are positive constants depending only on $p$ and $d>0$.
\end{lemma}
\begin{proof}
Multiplying the first equation of \eqref{KS-N} by $ u^{p-1}$ $(p>1)$ and recalling \eqref{li1}, we have
\begin{eqnarray}\label{2}
\begin{split}
\frac{1}{p} \frac{d}{dt} \int_{\Omega} u^p dx
= -(p-1) \int_{\Omega} \gamma(v)u^{p-2} |\nabla u|^2 dx+ (p-1) \int_{\Omega} \phi(v)u^{p-1} \nabla u\nabla v dx.
\end{split}
\end{eqnarray}
Then we multiply the second equation of \eqref{KS-N} by $ -\frac{p-1}{pd} u^p \phi(v)$ $ (p>1)$ to get
\begin{eqnarray}\label{3}
\begin{split}
0=& -\frac{p-1}{pd} \int_{\Omega}u^p \phi(v) (d\Delta v+u-v)dx\\
=& \frac{p-1}{p} \int_{\Omega} \nabla v(pu^{p-1}\nabla u \phi(v)+u^p\phi'(v)\nabla v)dx-\frac{p-1}{pd}\int_{\Omega} u^p \phi(v)(u-v)dx\\
=&(p-1)\int_{\Omega}\phi(v)u^{p-1}\nabla u\nabla v dx+\frac{p-1}{p}\int_{\Omega}u^p \phi'(v) |\nabla v|^2 dx\\
&-\frac{p-1}{pd}\int_{\Omega} u^p\phi(v)(u-v) dx.
\end{split}
\end{eqnarray}
Combining \eqref{2} with \eqref{3}, one has
\begin{eqnarray}\label{4}
\begin{split}
\frac{1}{p} \frac {d}{dt} \int_{\Omega} u^p dx =& -(p-1) \int_{\Omega} \gamma(v)u^{p-2} | \nabla u|^2 dx+2 (p-1) \int_{\Omega} \phi(v) u^{p-1} \nabla u \nabla v dx\\
&+ \frac {p-1}{p} \int_{\Omega} u^p\phi'(v) |\nabla v|^2 dx- \frac {p-1}{pd}\int_{\Omega} u^p \phi(v) (u-v)dx.
\end{split}
\end{eqnarray}
Let's define
 $$ A =(p-1) \gamma(v) >0, \ B= -(p-1) \phi(v) <0, \ C= -\frac{p-1}{p} \phi'(v) =\frac{p-1}{p} |\phi'(v)|>0$$
 and
$$ \vec{z_1} =u^{\frac {p}{2}-1} \nabla u, \ \vec{z_2}=u^{\frac {p}{2}}\nabla v.$$
Then \eqref{4} can be rewritten as
\begin{eqnarray} \label{5}
\begin{split}
\frac {1}{p}\frac {d}{dt}\int_{\Omega} u^p dx+\int_{\Omega} (A|\vec{z_1}|^2 +2B\vec{z_1}\vec{z_2}+C|\vec{z_2}|^2)dx
=-\frac {p-1}{pd}\int_{\Omega} u^p \phi(v)(u-v) dx.
\end{split}
\end{eqnarray}
Since $ A, C>0$, then
\begin{eqnarray*}
\begin{split}
A|\vec{z_1}|^2 +2B\vec{z_1}\vec{z_2} +C|\vec{z_2}|^2 \geq 0 &\iff B^2 -AC\leq0\\
&\iff p|\phi(v)|^2 \leq -\gamma(v)\phi'(v)=\gamma(v) |\phi'(v)|
\end{split}
\end{eqnarray*}
Under the hypothesis (H3),  we let $p$ be such that
\begin{eqnarray} \label{6}
\frac{n}{2}<p \leq\inf_{v \geq 0} \frac {\gamma(v)|\phi'(v)|}{|\phi(v)|^2}.
\end{eqnarray}
With \eqref{6}, if we define
\begin{eqnarray*}
\begin{aligned}
&\rho_1(v)=\frac {AC-B^2}{2C}=\frac{p-1}{2}\frac{|\phi(v)|^2}{|\phi'(v)|}\bigg(\frac {\gamma(v)|\phi'(v)|}{|\phi(v)|^2}-p\bigg),\\[1mm]
&\rho_2(v)=\frac {AC-B^2}{2A}=\frac{p-1}{2p}\frac{|\phi(v)|^2}{\gamma(v)}\bigg(\frac {\gamma(v)|\phi'(v)|}{|\phi(v)|^2}-p\bigg),
\end{aligned}
\end{eqnarray*}
then $\rho_1(v)\geq 0$ and $\rho_2(v)\geq 0$ for all $v\geq 0$ such that
$$ A|\vec{z_1}|^2 +2B \vec{z_1}\vec{z_2} +C|\vec{z_2}|^2 \geq \rho_1(v)|\vec{z_1}|^2+\rho_2(v)|\vec{z_2}|^2.$$
Thus it follows from \eqref{5} that
\begin{eqnarray*}\label{i3}
\begin{split}
&\frac{1}{p} \frac {d}{dt} \int_{\Omega} u^p dx +\int_{\Omega} (\rho_1(v){u^{p-2} |\nabla u|^2} +\rho_2(v) u^p|\nabla v|^2)dx\\
\leq &-\frac {p-1}{pd}  \int_{\Omega} u^p \phi(v)(u-v) dx\\
\leq &-\frac{p-1}{pd} \int_{\Omega} \phi(v)u^{p+1} dx+\frac{p-1}{pd} \int_{\Omega} u^p v\phi(v) dx.
\end{split}
\end{eqnarray*}
With the fact $\rho_i(v)\geq 0 (i=1,2)$, we add $\frac{1}{p}\int_\O u^pdx$ into the above inequality and obtain
\begin{eqnarray}\label{i3-1}
\begin{split}
\frac {d}{dt} \int_{\Omega} u^p dx+\int_\O u^pdx \leq -\frac{p-1}{d}  \int_{\Omega} \phi(v)u^{p+1}dx+\int_\O u^pdx+\frac{p-1}{d}\int_{\Omega} u^p v\phi(v)dx.
\end{split}
\end{eqnarray}
Owing to the Young's inequality, we have
\begin{eqnarray*}
\int_\O u^pdx \leq \frac{p-1}{2d}\int_\O \phi(v)u^{p+1}dx+c_1(p,d)\int_\O\phi(v)^{-p}dx
\end{eqnarray*}
which along with \eqref{i3-1} leads to
\begin{eqnarray}\label{i4n}
\begin{aligned}
\frac {d}{dt} &\int_{\Omega} u^p dx+\int_\O u^pdx +\frac{p-1}{2d}  \int_{\Omega} \phi(v)u^{p+1}dx\\
& \leq \frac{p-1}{d}\int_{\Omega} u^p v\phi(v)dx+c_1(p,d)\int_\O\phi(v)^{-p}dx.
\end{aligned}
\end{eqnarray}
Now we proceed to  estimate the first term on the right hand side of \eqref{i4n}.

{\bf Case 1} ($2\leq n\leq 3$). In this case, we employ Young's inequality to have
\begin{equation}\label{i3n}
\int_\O u^pv\phi(v)dx\leq \frac{1}{4}\int_\O\phi(v)u^{p+1}dx+c_2(p)\int_\O \phi(v)v^{p+1}dx
\end{equation}
where $c_2(p)=\frac{(4p)^p}{(p+1)^{p+1}}$. Then applying \eqref{i3n} into \eqref{i4n} yields
\begin{eqnarray}\label{i4}
\begin{aligned}
\frac {d}{dt} &\int_{\Omega} u^p dx+\int_\O u^pdx +\frac{p-1}{4d}  \int_{\Omega} \phi(v)u^{p+1}dx\\
& \leq c_1(p)\int_\O\phi(v)^{-p}dx+\frac{(p-1)c_2}{d}\int_\O \phi(v)v^{p+1}dx.
\end{aligned}
\end{eqnarray}
Thanks to the hypothesis (H2-(a)) and \eqref{lower}, we can find a constant $c_3>0$ so that $|\phi(v)|\leq c_3=\chi(\eta)$. Since $\frac{n}{2}<\frac{2}{n-2}$ for $n=2,3$, we can pick $p=\frac{n}{2}+\varepsilon$ with small $\varepsilon>0$ satisfying $p+1<\frac{n}{n-2}$. Therefore applying Lemma \ref{regularity} with the fact \eqref{L1}, we get a constant $c_5>0$ such that
$$\int_{\Omega} \phi(v)v^{p+1}dx  \leq c_3 \int_{\Omega} v^{p+1}dx  \leq c_4,  \ \mathrm{for} \ p=\frac{p}{2}+\varepsilon$$
which, upon a substitution into \eqref{i4}, yields a constant $c_5(p)$ such that
\begin{eqnarray*}\label{i5}
\begin{split}
\frac {d}{dt} \int_{\Omega} u^p dx+\int_\O u^pdx \leq c_5(p)+c_1(p)\int_\O\phi(v)^{-p}dx.
\end{split}
\end{eqnarray*}
This gives \eqref{pei}.

{\bf Case 2} ($n>3$). In this case, we employ the hypothesis (H2-(b)) and \eqref{lower} to find a constant $c_6>0$ such that $|v\phi(v)|<c_6$ and hence
\begin{eqnarray}\label{i6}
\begin{aligned}
\int_{\Omega} u^p v\phi(v)dx \leq c_6 \int_{\Omega} u^p dx &\leq \frac{1}{4}\int_\O \phi(v)u^{p+1}dx+c_7(p)\int_\O\phi(v)^{-p}dx
\end{aligned}
\end{eqnarray}
where the Young's inequality has been used and $c_7(p)>0$ is positive constant. Substituting \eqref{i6} into \eqref{i4n} yields a constant $c_8(p)=c_1(p,d)+\frac{(p-1)c_7}{d}>0$ such that
\begin{eqnarray*}\label{i7}
\frac{d}{dt} \int_{\Omega} u^p dx+\int_\O u^pdx +\frac{p-1}{4d}  \int_{\Omega} \phi(v)u^{p+1}dx
\leq c_8(p)\int_\O\phi(v)^{-p}dx
\end{eqnarray*}
which gives \eqref{pei}. The proof of Lemma \ref{priori} is completed.

\end{proof}

\subsection{Proof of main results} We are in a position to prove our main results.
\bigbreak

\noindent {\bf Proof of Theorem \ref{BS}}.
By Lemma \ref{criterion}, it remains only to show \eqref{Lp} holds. When $0<n<2$, \eqref{Lp} directly holds true by taking $p=1$ due to \eqref{L1}. Now we consider the case $n\geq 2$. Under the condition \eqref{condition}, we can find a constant $C_1>0$ from Lemma \ref{priori} such that
$$\frac{d}{dt} \int_{\Omega} u^p dx+\int_\O u^pdx<C_1.$$
This along with the Gronwall's inequality gives
$$\|u\|_{L^p(\O)}<C_2, \ \mathrm{for \ some}\ p>\frac{n}{2}$$
for some constant $C_2>0$. Then Theorem \ref{BS} follows immediately from Lemma \ref{criterion}.

\bigbreak

\noindent {\bf Proof of Theorem \ref{BSN}}. We consider the case of algebraically and exponentially decay  motility functions separately.

{\underline{Case 1 (algebraic decay)}}.  For convenience, we rewrite \eqref{rates1n} below
\begin{equation*}\label{rates1}
\gamma(v)=\frac{\sigma_1}{v^{\lambda_1}},\ \ \phi(v)=\frac{\sigma_2}{v^{\lambda_2}}, \ \sigma_1, \sigma_2>0, \lambda_1>0, \lambda_2>1.
\end{equation*}
Then the relation \eqref{dchi} is recovered when $\sigma_2=(1-\alpha)\sigma_1 \lambda_1$ and $\lambda_2=\lambda_1+1$.

Clearly the hypotheses (H1)-(H2) are satisfied. We next check the hypothesis (H3). Simple computation gives
$$\frac{\gamma(v) |\phi'(v)|}{|\phi(v)|^2}=\frac{\sigma_1 \lambda_2}{\sigma_2}v^{\lambda_2-\lambda_1-1}.$$
Hence the hypothesis (H3) with \eqref{lower} requires $\lambda_2 \geq \lambda_1+1$ and
\begin{equation}\label{hypo1}
\frac{\sigma_1 \lambda_2}{\sigma_2} >
\begin{cases}
\frac{n}{2}\ & \ \text{if} \ \lambda_2=\lambda_1+1,\\[2mm]
\frac{n}{2} \eta^{\lambda_1+1-\lambda_2}\ & \ \text{if} \ \lambda_2>\lambda_1+1.
 \end{cases}
\end{equation}
To get the global existence, it remains to verify the criterion \eqref{Lp}. We proceed with the following.

{\it Case a ($n=2$)}. When $n=2$, from Lemma \ref{regularity}, it clearly has that
\begin{equation}\label{e1}
\int_\O \phi(v)^{-p}dx=\sigma_2^{-p}\int_\O v^{\lambda_2 p}dx<\infty,\ \mathrm{for \ some} \ p>\frac{n}{2}.
\end{equation}
Then substituting \eqref{e1} into \eqref{pei} and using the Gronwall's inequality, we get \eqref{Lp} immediately.

{\it Case b ($n>2$)}. By the elliptic regularity theorem \cite{ADN-1959,ADN-1964} applied to the second equation of \eqref{KS-N}, we have $\|v\|_{W^{2,p}(\O)}\leq C_0\|u\|_{L^p(\O)}$ for some constant $C_0>0$, which along with the Sobolev embedding theorem yields
\begin{equation}\label{i12}
\|v\|_{L^\infty(\O)} \leq C_1\|u\|_{L^p(\O)}, \ \ \mathrm{for \ some}\  p>\frac{n}{2}
\end{equation}
with some constant $C_1>0$. Next we split the analysis into two cases. (1) If $\lambda_2<\frac{2}{n-2}$, then we can pick $p=\frac{n}{2}+\varepsilon$ with $0<\varepsilon<\frac{n}{\lambda_2(n-2)}-\frac{n}{2}$ such that $\lambda_2 p<\frac{n}{n-2}$, which together with \eqref{L1} and Lemma \ref{regularity} gives $\int_\O v^{\lambda_2 p}dx<C_2$ for some constant $C_2>0$. By the same argument as in {\it Case a}, we get \eqref{Lp}. (2) If $\lambda_2\geq \frac{2}{n-2}$,  we have $\lambda_2 p>\frac{n}{n-2}$ since $p>\frac{n}{2}$. Furthermore if we let $\lambda_2<\frac{n}{n-2}$, then $\frac{n}{2}(\lambda_2-1)<\frac{n}{n-2}$. Now choose $q>1$ such that  $\frac{n}{2}(\lambda_2-1)\leq q<\frac{n}{n-2}$, and one can check that $\theta=\lambda_2 p-q<p$ whenever $p>\frac{n}{2}$. Thus by the $L^p$-interpolation inequality, we have
$$\int_\O v^{\lambda_2 p}dx=\|v\|_{L^{\lambda_2 p}(\O)}^{\lambda_2 p} \leq \|v\|_{L^q(\O)}^q \|v\|_{L^\infty(\O)}^\theta.$$
This along with \eqref{i12}, Lemma \ref{regularity} with the fact $u\in L^1(\O)$ (see \eqref{L1}) as well as the Young's inequality gives
\begin{equation}\label{i13}
\int_\O \phi(v)^{-p}dx=\sigma_2^{-p} \int_\O v^{\lambda_2p}dx\leq C_3\|u\|_{L^p(\O)}^\theta\leq C_4+\frac{1}{2}\|u\|_{L^p(\O)}^p
\end{equation}
for some constants $C_3, C_4>0$. Then substituting \eqref{i13} into \eqref{pei} yields a constant $C_5>0$ such that
\begin{eqnarray*}\label{i11n}
\frac{d}{dt} \int_{\Omega} u^p dx+\frac{1}{2}\int_\O u^pdx \leq  C_5
\end{eqnarray*}
which again by the Gronwall's inequality gives \eqref{Lp}. In summary,  with \eqref{hypo1} we get \eqref{Lp} for any $0<\lambda_2<\frac{n}{n-2}$. Noticing that $\lambda_2<\frac{n}{n-2} (n\geq 2)$ is equivalent to $\frac{\lambda_2}{\lambda_2-1}>\frac{n}{2}$, and combining with \eqref{hypo1}, we get the condition \eqref{con1} for the global existence of solutions to \eqref{KS-N} with \eqref{rates1n}.
This finishes the proof for { Case 1}.

{\underline{Case 2 (exponential decay)}}. For convenience, we recast \eqref{rates2n} as follows
\begin{equation*}\label{rates2}
\gamma(v)=\exp(-\chi_1 v),\ \ \phi(v)=\delta \exp(-\chi_2 v), \ \ \ \chi_1>0, \chi_2>0.
\end{equation*}
By a direct computation, we have
$$\frac{\gamma(v) |\phi'(v)|}{|\phi(v)|^2}=\frac{\chi_2}{\delta} \exp((\chi_2-\chi_1)v)$$
which subject to hypothesis (H3) and \eqref{lower} impose the conditions on $\chi_i (i=1,2)$ as
\begin{equation*}\label{hypo2}
\chi_2\geq \chi_1, \ \mathrm{and} \
\chi_2>
\begin{cases}
\frac{n\delta}{2}\ & \ \text{if} \ \chi_1=\chi_2\\[2mm]
\frac{n\delta}{2} \exp\{(\chi_1-\chi_2)\eta\}\ & \ \text{if} \ \chi_1<\chi_2.
 \end{cases}
\end{equation*}
Next we only need to estimate $\int_\O\phi(v)^{-p}dx =\delta^{-p} \int_\O e^{\chi_2 p v}dx$. In this scenario, we focus on the case $n\leq 2$ and the case $n>3$ is still open.

When $n<2$, we have $\|\nabla v\|_{L^n(\O)} \leq C\|u\|_{L^1(\O)}=C\|u_0\|_{L^1(\O)}$ (cf. \cite[(2.11)]{Nagai-2001}). Noticing that $\|v\|_{L^1(\O)}$ is obtained directly by integrating the second equation of \eqref{KS-N}
\begin{equation*}\label{massv}
\|v\|_{L^1(\O)}=\|u\|_{L^1(\O)}=\|u_0\|_{L^1(\O)}.
\end{equation*}
Then by the Trudinger-Moser inequality (see Lemma \ref{moser}), we obtain that
\begin{equation}\label{mosern0}
\int_\O\phi(v)^{-p}dx =\delta^{-p}\int_\O \exp(\chi_2 p v)dx \leq c_0\exp(c_1\|\nabla v\|_{L^n}^n+c_2\|v\|_{L^1}^n)<\infty, \ n<2
\end{equation}
for some constant $c_i (i=0,1,2)$ depending on $n, p, \chi_2$.

When $n=2$, we let $p=\frac{n}{2}+\varepsilon=1+\varepsilon$ with $0<\varepsilon <\frac{\frac{4\pi \d}{m}-\chi_2}{\chi_2}$ under the assumption $\chi_2<\frac{4\pi \d}{m}$. Then we have $\chi_2 p=\chi_2(1+\varepsilon)<\frac{4\pi \d}{m}$ and hence it follows from Lemma \ref{Merle} that
\begin{equation}\label{mosern1}
\int_\O\phi(v)^{-p}dx =\delta^{-p}\int_\O \exp(\chi_2 p v)dx <\infty, \ n=2.
\end{equation}
Feeding \eqref{pei} on \eqref{mosern0} or \eqref{mosern1} and applying the Gronwall's inequality, we have $\|u\|_{L^p(\O)}\leq c_5$ for some $p>\frac{n}{2}$. This along with Lemma \ref{criterion} finishes the proof of {Case 2} and hence of Theorem \ref{BSN}.

\bigbreak

\noindent {\bf Proof of Theorem \ref{BSNN}}. We consider two cases separately.

(i) If $\phi(v)=(\alpha-1)\gamma'(v)$ with $\gamma(v)=\frac{\sigma}{v^\lambda}$, which is a particular case of \eqref{rates1n} with $\lambda_2=1+\lambda_1, \sigma_2=(1-\alpha)\lambda_1 \sigma_1, \lambda_1=\lambda, \sigma_1=\sigma$. Then the condition \eqref{con1} becomes
\begin{equation}\label{in1}
\frac{\lambda}{1+\lambda} \cdot \frac{n}{2}<\min\Big\{1,\frac{1}{1-\alpha}\Big\}.
\end{equation}
If $\alpha<0$, then \eqref{in1} $\Leftrightarrow \frac{\lambda}{1+\lambda} \cdot \frac{n}{2}<\frac{1}{1-\alpha} \Leftrightarrow \lambda<\frac{2}{n(1-\alpha)-2}$. While if $0\leq \alpha<1$, then \eqref{in1} $\Leftrightarrow \frac{\lambda}{1+\lambda} \cdot \frac{n}{2}<1 \Leftrightarrow \lambda<\frac{2}{n-2}$. This gives \eqref{gcon} and hence completes the proof of case (i).

(ii) If $\phi(v)=(\alpha-1)\gamma'(v)$ with $\gamma(v)=e^{-\chi v}$, which corresponds to $\chi_2=\chi_1=\chi, \delta=(1-\alpha)\chi$ in \eqref{rates2n}. Then the condition $\frac{n\delta}{2} \exp\{(\chi_1-\chi_2)\eta\}<\chi_2$ in \eqref{con2} with $n=2$ requires $0<\alpha<1$. This along with the condition $\chi_2<\frac{4\pi d}{m}$ completes the proof of case (ii).

\section{Stationary solutions}
In this section, we shall explore the non-constant stationary solutions to the Keller-Segel system \eqref{KS-N} with \eqref{dchi}. First notice that the cell mass is conserved in the time-dependent problem, see \eqref{L1}. Hence  the relevant stationary problem reads as
\begin{eqnarray}\label{KS-S}
\begin{cases}
\nabla \cdot (\gamma(v)\nabla u-u \phi(v)\nabla v)=0,  &x\in\Omega,\\
\d\Delta v+u-v=0, &x\in\Omega, \\
\partial_{\nu}u=\partial_{\nu} v=0 &x\in \partial \Omega,\\
\int_\O u(x)dx=m
\end{cases}
\end{eqnarray}
where $m>0$ is a constant denoting the cell mass and
\begin{equation}\label{dchin}
\phi(v)=\beta \gamma'(v), \ \ \beta=\alpha-1.
\end{equation}
Substituting \eqref{dchin} into \eqref{KS-S}, we find that the first equation of \eqref{KS-S} may be written as
\begin{equation}\label{KS-S1}
\nabla \cdot \Big(u\gamma(v)\nabla \ln \frac{u}{\gamma(v)^\beta}\Big)=0.
\end{equation}
Multiplying \eqref{KS-S1} by $\ln\frac{u}{\gamma(v)^\beta}$ and integrating the resulting equations by parts along with the Neumann boundary conditions, we get
$$\int_\O u\gamma(v)\Big|\nabla \ln \frac{u}{\gamma(v)^\beta}\Big|^2dx=0$$
which immediately yields $\nabla \ln \frac{u}{\gamma(v)^\beta}=0$ and hence
$$u(x)=\theta \gamma(v)^\beta $$
where $\theta>0$ is a constant. With the mass constraint given in the fourth equation of \eqref{KS-S}, we integrate the above equation and get
$$\theta=\frac{m}{\int_\O \gamma(v)^\beta dx}.$$
We thus reduce the stationary system \eqref{KS-S} into a non-local semi-linear problem
\begin{eqnarray}\label{KS-S2}
\begin{cases}
\D \d \Delta v -v+\frac{m}{\int_\O \gamma(v)^{\beta} dx}\gamma(v)^{\beta}=0, &x\in\Omega, \\
\partial_{\nu} v=0 &x\in \partial \Omega
\end{cases}
\end{eqnarray}
with
$$u(x)=\frac{m}{\int_\O \gamma(v)^\beta dx} \gamma(v)^\beta.$$
In order to get some specific results, we need to specify the form of $\gamma(v)$ for which we consider two cases: algebraically and exponentially decay functions. We have the following results.

\begin{theorem}\label{th-ss}
Let $\alpha<1$. Then the following results hold.
\begin{itemize}
\item[(a)] Consider $\gamma(v)=\frac{\sigma}{v^\lambda} (\sigma, \lambda>0)$. If $(1-\alpha)\lambda>1$ when $n=1,2$ and $1<(1-\alpha)\lambda< \frac{n+2}{n-2}$ when $n\geq 3$,  there are constants $0<d_0<d_1$ depending on the domain $\Omega$ such that \eqref{KS-S2} admits a non-constant solution whenever $d<d_0$ and only constant solution if $d>d_1$. If $0<(1-\alpha)\lambda\leq 1$, then the constant $v=\frac{m}{|\Omega|}$ is the only nonnegative solution of \eqref{KS-S2} for any $d>0$.
\item[(b)] Consider $\gamma(v)=e^{-\chi v} (\chi>0)$ and let $\O$ be a disc in $\R^2$. Then the problem \eqref{KS-S2} admits a non-constant radial solution if $m>\frac{8\pi d }{\chi (1-\alpha)}$, while the constant $v=\frac{m}{|\Omega|}$ is the only radial solution to \eqref{KS-S2} if $m<\frac{8\pi d}{\chi (1-\alpha)}$.
\end{itemize}
\end{theorem}

\subsection{Motility with algebraic decay} Assuming $\gamma(v)=\frac{\sigma}{v^\lambda} (\sigma, \lambda>0)$, the stationary problem \eqref{KS-S2} becomes
\begin{eqnarray}\label{KS-S3}
\begin{cases}
\D \d\Delta v -v+\frac{m}{\int_\O v^{k}  dx}v^{k}=0, &x\in\Omega, \\
\partial_{\nu} v=0 &x\in \partial \Omega
\end{cases}
\end{eqnarray}
where assume that $k=-\lambda \beta=\lambda(1-\alpha)>0$. To the best of our knowledge, the existence of non-trivial solutions to the non-local problem \eqref{KS-S3} was missing in the literature. Below we shall show the existence of solutions to \eqref{KS-S3} via the following localized problem
\begin{eqnarray}\label{KS-S4}
\begin{cases}
\D \d \Delta w -w+ w^{k}=0, &x\in\Omega, \\
\partial_{\nu} w=0, &x\in \partial \Omega
\end{cases}
\end{eqnarray}
which has been widely studied in the literature (cf. \cite{Lin-Ni1, Ni, Ni-T1, Ni-T2, Ni-T3}). The most prominent feature of \eqref{KS-S3} is that its solutions possess point condensation phenomena meaning that the solutions aggregate at finite number of points and tend to zero elsewhere as $\d \to 0$. Moreover when $\d$ is small, \eqref{KS-S4}  has a non-constant least energy solution which has exactly one local maximum  on the boundary and is considered to be the most stable one among all possible non-constant solutions. We cite the following well-known results (cf. \cite{Lin-Ni1, YK-AAM-2017}).

\begin{lemma}\label{existence-w}
Let $k>1$ if $n=1,2$ and $1<k< \frac{n+2}{n-2}$ if $n\geq 3$. Then there are constants $0<d_0<d_1$ depending on the domain $\Omega$ such that \eqref{KS-S4} admits a non-constant solution whenever $d<d_0$ and only constant solution if $d>d_1$. If $0<k\leq 1$, the constant $w=1$ is the only nonnegative solution to \eqref{KS-S3} for any $d>0$.
\end{lemma}

Now we are in a position to prove Theorem \ref{th-ss}(a).

\medskip

\noindent {\it Proof of Theorem \ref{th-ss}(a)}.  Let $w$ be a solution of \eqref{KS-S4} with  $\int_\O wdx=m_0$. If $m_0=m$, then $w$ is a solution of \eqref{KS-S3} since $\int_\O wdx=\int_\O w^{k}dx$ by the integration of  \eqref{KS-S4}. Otherwise, if $m_0\ne m$, we define
$$\V=\frac{m}{m_0}w.$$
Then $\int_\O Vdx=m$ and from \eqref{KS-S4}, we may check that $V$ satisfies
\begin{eqnarray}\label{KS-S5}
\begin{cases}
\D \d \Delta \V -\V+\Big(\frac{m}{m_0}\Big)^{k-1}\V^{k}=0, &x\in\Omega, \\
\partial_{\nu} \V=0 &x\in \partial \Omega.
\end{cases}
\end{eqnarray}
On the other hand, integrating \eqref{KS-S5} yields
 that $\int_\O \V dx=\int_\O \big(\frac{m}{m_0}\big)^{k-1}\V^{k}dx=m$. Then
\begin{equation}\label{KS-S6}
\Big(\frac{m}{m_0}\Big)^{k-1}=\frac{m}{\int_\O \V^kdx}.
\end{equation}
With \eqref{KS-S6} and \eqref{KS-S5}, we see that $V=\frac{m}{m_0}w$ is a solution to \eqref{KS-S3}. With $k=\lambda \beta=(1-\alpha)\lambda$ and existence results in Lemma \ref{existence-w} for $w$, we get the existence of solutions to \eqref{KS-S3} and hence prove the first part of Theorem \ref{th-ss}(a). We proceed to prove that \eqref{KS-S3} has only constant solution if $0<(1-\alpha)\lambda\leq 1$ (namely $0<k\leq1$). Arguing by contradiction, we assume that there is a non-constant solution to \eqref{KS-S3} in the case of $0<k\leq 1$. Then $v$ is a solution of the following problem
\begin{eqnarray*}\label{KS-S3-1}
\begin{cases}
\D \d\Delta v -v+\xi v^{k}=0, &x\in\Omega, \\
\partial_{\nu} v=0 &x\in \partial \Omega
\end{cases}
\end{eqnarray*}
with $\xi=\frac{m}{\int_\O v^{k} dx}$. A direct calculation will show $w=\xi^{\frac{1}{k-1}} v$ is also a (non-constant) solution to \eqref{KS-S4}, which contradicts the results of Lemma \ref{existence-w}. This completes the proof of Theorem \ref{th-ss}(a).

\subsection{Motility with exponential decay}
Now we consider $\gamma(v)=e^{-\chi v} (\chi>0)$, which turns the stationary problem \eqref{KS-S2} to be
\begin{eqnarray}\label{KS-S3n}
\begin{cases}
\D\d\Delta v -v+\frac{m}{\int_\O e^{-\chi \beta v}  dx}e^{-\chi \beta v}=0, &x\in\Omega, \\
\partial_{\nu} v=0, &x\in \partial \Omega.
\end{cases}
\end{eqnarray}
With a change of variable
$$\tilde{v}=-\chi \beta v, \ \tilde{m}=-\chi \beta m,$$
we can transform \eqref{KS-S3n} into the following problem
\begin{eqnarray}\label{KS-S3n1}
\begin{cases}
\D\d\Delta \tilde{v} -\tilde{v}+\frac{\tilde{m}}{\int_\O e^{\tilde{v}}  dx}e^{\tilde{v}}=0, &x\in\Omega, \\
\partial_{\nu} \tilde{v}=0 &x\in \partial \Omega.
\end{cases}
\end{eqnarray}
The analysis of the nonlocal problem \eqref{KS-S3n1} is delicate and the geometry of domain $\Omega$ plays a role in determining the existence of solutions. It was proved in \cite{Senba-Suzuki-2000} that in two dimensions \eqref{KS-S3n1} only admits constant solution if $0<\tilde{m}\ll 1$  while admits non-constant solutions if $\tilde{m}>4\pi d$ and $\tilde{m}\ne 4k\pi d$ for $k=1,2,\cdots$. Similar results were obtained in \cite{Wang-Wei}. If $\Omega$ has some special geometry, non-constant solutions may also exist for $\tilde{m}<4\pi d$ (see \cite{Senba-Suzuki-2000}). When $\tilde{m}$ is sufficiently close to $4k\pi d(k=1,2,\cdots)$ in two dimensions, blow-up solutions may exist, (cf. \cite{delPino-Wei}), while in three or higher dimensions, blow-up solutions may exist for any $\tilde{m}>0$ (cf. \cite{Agudelo}). For the radial symmetric case, the following result (cf. \cite[Theorem 4]{Senba-Suzuki-2000}) gives a threshold of mass in two dimensions.

\begin{lemma}\label{existence-n}
Let $\O$ be a disc in $\R^2$ and $\tilde{v}(x)=\tilde{v}(|x|)$. Then the problem \eqref{KS-S3n1} admits a non-constant if $\tilde{m}>8\pi d$, while admits only constant solution $\tilde{v}=\tilde{m}{|\Omega|}$ if $m<8\pi d$.
\end{lemma}

We remark for the radially symmetric domain $\Omega \subset \R^2$, if the solution is only required to be constant on the boundary (not necessarily radially symmetric), it was shown in \cite{WWY} that \eqref{KS-S3n1} only admits a unique constant solution.
\vspace{0.2cm}

\noindent {\it Proof of Theorem \ref{th-ss}(b)}. Noticing that $\tilde{m}=\chi(1-\alpha)m$, we obtain Theorem \ref{th-ss}(b) immediately as a consequence of Lemma \ref{existence-n}.


\section{Summary and discussion}
In this paper, we consider the parabolic-elliptic Keller-Segel system \eqref{KS} with $\tau=0$, where both cell diffusion rate $\gamma(v)$ and chemotactic coefficient $\phi(v)$ are functions of the signal density. The prototypical relation between $\gamma(v)$ and $\phi(v)$ was given by \eqref{dchi} in \cite{KS-1971-JTB2}. Although system \eqref{KS} has been proposed almost 50 years, the mathematical results are still very limited when both $\gamma(v)$ and $\phi(v)$ are non-constant. The existing results were developed only for the special case $\alpha=0$, namely $\phi(v)=-\gamma'(v)$, for which the system  \eqref{KS} was substantially reduced to \eqref{KS1}. By far no results have been available for the case $\alpha\ne0$ or general functions $\gamma(v)$ and $\phi(v)$. This paper takes a step forward to find suitable conditions on $\gamma(v)$ and $\phi(v)$ (see hypotheses (H1)-(H3) and \eqref{condition}) for the global boundedness of solutions in a smooth boundary domain of any dimension with Neumann boundary conditions (see Theorem \ref{BS}). These conditions include but have gone beyond the relation \eqref{dchi}. As an application, we give  examples for motility functions with algebraic and exponential decay and transform these conditions to the decay rates (see Theorem \ref{BSN}). By the results of Theorem \ref{BSN}, we obtain the global boundedness of solutions to \eqref{KS-N} with relation \eqref{dchi} for $\alpha\ne 0$ (see Theorem \ref{BSNN}).  Lastly we give some results on the existence/nonexistence of non-constant stationary solutions of \eqref{KS-N} with \eqref{dchi} $\gamma(v)$ with algebraic or exponential decay.

The results in the present paper with existing results in \cite{FujieJiang2020-1, Jin-Wang-PAMS, AY-Nonlinearity-2019} demonstrate that depending on the decay rate of non-constant motility function $\gamma(v)$ and $\phi(v)$, the relation between $\gamma(v)$ and $\phi(v)$ and space dimensions,  the Keller-Segel system \eqref{KS} with has very rich dynamics/patterns such as global boundedness, blow up, condensation  patterns and so on. Although some progresses have been made in this paper along with above-mentioned works,  there are many interesting questions left open.  First the asymptotic behavior of solutions is not explored in this paper, which will be an intricate problem given the wealthy behavior of stationary solutions as shown in Section 4.  Moreover the analytical tools tackling \eqref{KS} and its special case \eqref{KS1} may be very different. For example, the comparison principle is applicable to the simplified \eqref{KS1} by some technical treatment as done in \cite{FujieJiang2020-1}. However the method of \cite{FujieJiang2020-1} essentially depends on the structure of \eqref{KS1} and may not be directly applicable to the general case model \eqref{KS} for which comparison principle fails in general due to the cross diffusion.  Even for the simplified Keller-Segel system \eqref{KS1}, its understanding is far from being complete in spite of some progresses made recently in \cite{FujieJiang2020-1, Jin-Wang-PAMS}. For example in higher dimensions ($n\geq 3$), the global dynamics of solutions to \eqref{KS1} is unknown for exponential decay $\gamma(v)=\exp(-\chi v)$ or algebraic decay $\gamma(v)=\frac{\sigma}{v^k}$ with  $k\geq \frac{2}{n-2}$.  Turning to the Keller-Segel system \eqref{KS} with non-constant $\gamma(v)$ and $\phi(v)$, the present paper establishes the global boundedness of solutions for the parabolic-elliptic case model \eqref{KS-N} under conditions (H1)-(H3) with \eqref{condition}, which cover a wide range of motility functions $\gamma(v)$ and $\phi(v)$.  Whether these results can be extended to the parabolic-parabolic case model (i.e. \eqref{KS} with $\tau=1$) remains open.
In particular the global dynamics of solutions for exponentially decay motility functions in three or higher dimensions still remain poorly understood.
The hypotheses (H1)-(H3) plus \eqref{condition} prescribe sufficient conditions for the global boundedness of solutions. But to what extend these conditions are necessary is obscure. An immediate relevant question is whether solutions blow up if some (or all) of these conditions fail. The answer seems elusive since the global dynamics of solutions may critically depend on the decay rate of $\gamma(v)$ and $\phi(v)$ and space dimensions as can be seen from the specialized model \eqref{KS1}. The results of Theorem \ref{BSNN} apply to the case $\alpha<1$ only, while the results for $\alpha>1$ remains open. By the relation \eqref{dchi}, we see that $\alpha=1$ is a critical number determining the sign of $\phi(v)$. When $\alpha>1$, the Keller-Segel system will become a repulsive chemotaxis model if $\gamma'(v)<0$. This is opposite to the attractive case ($\alpha<1$) that we explore in this paper. Therefore it is worthwhile to study the case $\alpha>1$ for \eqref{KS}-\eqref{dchi} to examine how the dynamics will be different from the attractive case $\alpha<1$. Though the foregoing questions are by no means exhaustive ones open for the Keller-Segel system \eqref{KS}, the answer of these questions will  certainly enhance the understanding of the immensely rich dynamics encompassed in the Keller-Segel system \eqref{KS} with non-constant motility functions $\gamma(v)$ and $\phi(v)$.

\bigbreak
\noindent {\bf Acknowledgement}. The author is grateful to Prof. Benoit Perthame from Sorbonne Universit\'{e} for very inspiring discussions on the topic of this paper when he visited the Hong Kong Polytechnic University. Thanks are also given to Prof. Haiyang Jin from South China University of Technology and Prof. Wen Yang from Wuhan Institute of Physics and Mathematics of Chinese Academy of Sciences for many insightful discussions during the preparation of  this paper. This research is partially supported by the Hong Kong RGC GRF grant No. 15303019 (Project Q75G).

\end{document}